\documentclass[11pt]{article}%
\usepackage[letterpaper, portrait, margin=1in, headheight=0pt]{geometry}
\usepackage{mathtools}
\usepackage{enumerate}
\usepackage{amsmath,amsthm,amssymb,amsfonts}
\usepackage{enumitem}
\usepackage{mathrsfs}
\usepackage{tikz-cd}
\usepackage{multicol}
\usepackage[title]{appendix}
\usetikzlibrary{quotes,angles}

\newcommand{\F}{\mathbb{F}}

\newcommand{\inv}{^{-1}}

\newcommand{\cH}{\mathcal{H}}

\newtheorem{thm}{Theorem}[subsection]
\newtheorem{lem}[thm]{Lemma}

\newtheorem{cor}[thm]{Corollary}

\theoremstyle{definition}

\theoremstyle{remark}

\title{Characterizing Nilpotent Associative Algebras by Their Multiplier}
\author{Erik Mainellis}
\date{}

\begin{document}

\maketitle

\begin{abstract}
    The paper concerns an analogue of the famous Schur multiplier in the context of associative algebras and a measure of how far its dimension is from being maximal. Applying a methodology from Lie theory, we characterize all finite-dimensional nilpotent associative algebras for which this measure is ten or less.
\end{abstract}

\section{Introduction}
Schur multipliers are, roughly, a way of viewing the second (co)homology group $\cH^2(L)$ of an algebraic object $L$ as the kernel of a stem extension of maximal dimension. This theory has been developed for Lie \cite{batten}, Leibniz \cite{mainellis batten}, diassociative \cite{mainellis batten di}, and triassociative \cite{mainellis batten tri} algebras. In particular, the work of \cite{mainellis batten di} initiates the study of multipliers for associative algebras as a consequence of the diassociative generalization. This theory is furthered in \cite{mainellis nilp mult}, where the author focuses specifically on multipliers of nilpotent diassociative algebras and also discusses the associative specialization of the results. These results concern dimension bounds on the multiplier that form a basis for the work of the present paper.

There has been great success in characterizing nilpotent Lie and Leibniz algebras by invariants related to the dimension of their multipliers (see \cite{arab, batten mult, edal mult new, edal mult, hardy, hardy stitz, shamsaki}). Generally, these arguments consider a measure of how far the dimension of the multiplier $M(L)$ is from being maximal and proceed to classify algebras based on this distance. For a Lie algebra $L$ of dimension $n$, the measure is $\frac{1}{2}n(n-1) - \dim M(L)$. For a Leibniz algebra, it is $n^2 - \dim M(L)$. Some variations have been considered.

The objective of the present paper is to classify nilpotent associative algebras based on a similar measure. In particular, we define $t(A) = n^2 - \dim M(A)$ for a nilpotent associative algebra $A$ of dimension $n$ and determine all $A$ such that $t(A) \leq 10$. Throughout, we work over the complex field and assume that all algebras are finite-dimensional. We denote by $A(n)$ the abelian algebra of dimension $n$ and by $A\ast B$ the central sum of associative algebras $A$ and $B$. We let $E(n)$ range over all extra special algebras of dimension $n$.

\section{Preliminaries}

We refer the reader to the preliminaries of \cite{mainellis extra special} for the notions of \textit{subalgebra}, \textit{ideal}, \textit{central sum}, \textit{extension}, \textit{section}, \textit{central extension}, and \textit{stem extension} for associative algebras. Let $A$ be an associative algebra. We denote by $A'=AA$ the \textit{derived ideal} of $A$, the ideal generated by all products in $A$. We define the \textit{center} of $A$ in the Lie sense; it is the ideal consisting of all $z\in A$ such that $za = az = 0$ for all $a\in A$. Also in the Lie sense, $A$ is called \textit{abelian} if all products are zero (in other words, if $Z(A) = A$, or if $A'=0$). An algebra $A$ is called \textit{nilpotent} if there exists a natural number $n$ such that any product of $n$ elements in $A$ is zero.

Given an associative algebra $A$, a pair of associative algebras $(K,M)$ is called a \textit{defining pair} for $A$ if $A\cong K/M$ and $M\subseteq Z(K)\cap K'$. Equivalently, a defining pair describes a stem extension \[0\xrightarrow{} M\xrightarrow{} K\xrightarrow{\omega} A\xrightarrow{} 0\] where $M = \ker \omega$. We say that a defining pair $(K,M)$ is a \textit{maximal defining pair} if the dimension of $K$ is maximal. In this case, $K$ is called the \textit{cover} of $A$ and $M$ is called the \textit{multiplier} of $A$, denoted by $M(A)$. It is known that $M(A)\cong\cH^2(A,\F)$, the second cohomology group with coefficients in the base field $\F$, and that covers are unique (see \cite{mainellis batten di}). As in the Leibniz case, the dimension of the multiplier $M(A)$ for an associative algebra $A$ is bounded by $(\dim A)^2$. In \cite{mainellis batten tri}, the author gives a table comparing the Lie, Leibniz, associative, diassociative, and triassociative algebra cases. We thus define a measure \[t(A) = (\dim A)^2 - \dim M(A)\] of how far $\dim M(A)$ is from being maximal. It is clear that $\dim M(A)=(\dim A)^2$ if and only if $A$ is abelian, and so $t(A) = 0$ in this case.

The author has not been able to find a direct proof of the following Künneth-Loday style formula for associative algebras, but it is easily provable from first principles via the methodology of its Lie analogue (see Theorem 1 in \cite{batten mult}) with appropriate substitutions. We note that a Leibniz version of this formula is used in \cite{edal mult} for the effort of characterizing nilpotent Leibniz algebras by their multipliers. Given finite-dimensional associative algebras $A$ and $B$, \begin{align}\label{kunneth}
    \dim M(A\oplus B) = \dim M(A) + \dim M(B) + 2\dim(A/A'\otimes B/B').
\end{align}
This formula is useful for proving the following Lemma, on which we will rely heavily for the proof of the main result.

\begin{lem}\label{ideal equality}
Let $A$ be a nilpotent, finite-dimensional associative algebra such that $Z(A)\not\subseteq A'$. Then there exists a 1-dimensional ideal $Z$ in $A$ such that $A = I\oplus Z$ and $t(I) + 2\dim(I') = t(A)$ for an ideal $I$ in $A$.
\end{lem}

\begin{proof}
Under the given assumptions, there exists a 1-dimensional subspace $Z\subseteq Z(A)$ such that $Z\not\subseteq A'$. Let $I$ be the complement to $Z$ in $A$ such that $A=I\oplus Z$ and $A'\subseteq I$. We note that $\dim M(Z) = (1)^2 = 1$ since $Z$ is abelian and that $\dim(I/I'\otimes Z/Z') = \dim I - \dim(I')$ since $Z/Z' = Z$. By (\ref{kunneth}), we have \begin{align*}
    \dim M(I\oplus Z) = \dim M(I) + \dim M(Z) + 2\dim(I/I'\otimes Z/Z')
\end{align*} which yields \begin{align*}
    n^2 - t(A) &= (n-1)^2 - t(I) + 1 + 2(n-1-\dim(I'))
\end{align*} for $n=\dim A$. Simplifying this equation, we obtain $t(I) + 2\dim(I') = t(A)$.
\end{proof}

Let $A$ be a nilpotent, finite-dimensional associative algebra. The following dimension bounds on the multiplier $M(A)$ of $A$ were obtained in \cite{mainellis nilp mult} for the more general case of diassociative algebras. We state their associative versions here: \begin{align}\label{bound i}
    \dim M(A) + 1 \leq \dim M(A/Z) + 2\dim(A/A')
\end{align} for any 1-dimensional ideal $Z\subseteq Z(A)\cap A'$, and \begin{align}\label{bound iii}
    \dim M(A)\leq \dim M(A/A') + \dim(A')[2\dim(A/A') - 1].
\end{align}
We will apply inequality (\ref{bound i}) directly in the proof of the main theorem. From (\ref{bound iii}), we relate $t(A)$ to $\dim(A')$ in the following manner.

\begin{lem}\label{derived bound}
Let $A$ be a nilpotent, finite-dimensional associative algebra. Then \[t(A)\geq \dim(A')(\dim(A')+1).\]
\end{lem}

\begin{proof}
    Let $n=\dim A$ and $m=\dim(A')$. Since $A/A'$ is abelian, we have $\dim M(A/A') = (n-m)^2$. By (\ref{bound iii}), we compute \begin{align*}
        \dim M(A) &\leq (n-m)^2 + m(2(n-m) - 1) \\ &= n^2 - m^2 - m
    \end{align*} which yields $t(A) \geq n^2 - n^2 + m^2 + m = m(m+1)$.
\end{proof}

Finally, we say that an associative algebra $A$ is \textit{extra special} if $Z(A) = A'$ and $\dim(Z(A)) = 1$. In \cite{mainellis extra special}, the author obtained the classification of these algebras as well as of their multipliers. In particular, their structure is precisely the same as that of the Leibniz case (obtained in \cite{edal}), but their multipliers are different in a handful of cases. We state two theorems from \cite{mainellis extra special}.

\begin{thm}
Any extra special associative algebra is a central sum of the following five classes of extra special associative algebras:
\begin{enumerate}
    \item[i.] $J_1$ with basis $\{x,z\}$ and nonzero product $xx=z$;
    
    \item[ii.] $J_n$ for $n=2,3,\dots$, with basis $\{x_1,\dots,x_n,z\}$ and nonzero products \begin{align*}
        x_1x_2 = z, ~~~ x_2x_3 = z, ~~~ \cdots ~~~ x_{n-1}x_n = z;
    \end{align*}
    
    \item[iii.] $\Gamma_n$ for $n=2,3,\dots$, with basis $\{x_1,\dots,x_n,z\}$ and nonzero products \begin{align*}
        & x_nx_1 = z, ~~~ x_{n-1}x_2 = -z, ~~~ \cdots ~~~ x_ix_{n-i+1} = (-1)^{n-i+2}z, ~~~\cdots ~~~x_2x_{n-1} = (-1)^nz, \\ & x_nx_2 = z, ~~~ x_{n-1}x_3 = -z, ~~~ \cdots ~~~ x_ix_{n-i+2} = (-1)^{n-i+2}z, ~~~\cdots ~~~x_2x_{n} = (-1)^nz, \\ & x_1x_n = (-1)^{n+1}z;
    \end{align*}

    \item[iv.] $H_2(\lambda)$ with basis $\{x_1,x_2,z\}$ and nonzero products $x_1x_2 = z$, $x_2x_1 = \lambda z$ for $0\neq \lambda \neq 1$;

    \item[v.] $H_{2n}(\lambda)$ for $n=2,3,\dots$, with basis $\{x_1,\dots,x_{2n},z\}$ and nonzero products \begin{align*}
        & x_1x_{n+1} = z, ~~~ x_2x_{n+2}=z, ~~~ \cdots ~~~ x_nx_{2n} = z, \\ & x_{n+1}x_1 = \lambda z, ~~~ x_{n+2}x_2=\lambda z, ~~~ \cdots ~~~ x_{2n}x_n = \lambda z, \\ & x_{n+1}x_2 = z, ~~~ x_{n+2}x_3 = z, ~~~ \cdots ~~~ x_{2n-1}x_n = z,
    \end{align*} where $0\neq \lambda \neq (-1)^{n+1}$.
\end{enumerate} Here, $\lambda\in \F$ is determined up to replacement by $\lambda\inv$.
\end{thm}

\begin{thm}
    Let $A$ be an extra special associative algebra. Then $\dim M(A) = (\dim(A) - 1)^2-1$ with the exception of $A=J_1$. In particular, $\dim M(J_1) = 1$.
\end{thm}

\begin{cor}
Let $A$ be an extra special associative algebra. Then $t(A) = 2\dim A$ with the exception of $A=J_1$. In particular, $t(J_1) = 3$.
\end{cor}

\begin{proof}
    For all extra special $A$ besides $J_1$, we compute \begin{align*}
        t(A) &= n^2 - \dim M(A) \\ &= n^2 - (n-1)^2 + 1 \\ &= 2n
    \end{align*} where $n = \dim A$. For $J_1$, we compute $t(J_1) = 2^2 - 1 = 3$.
\end{proof}

\section{The Main Result}

\begin{thm}
    Let $A$ be a complex nilpotent associative algebra. Then
    \begin{enumerate}
        \item[i.] $t(A) = 0$ if and only if $A$ is abelian;
        \item[ii.] there is no $A$ such that $t(A) = 1$ or $t(A)=2$;
        \item[iii.] $t(A) = 3$ if and only if $A = J_1$;
        \item[iv.] there is no $A$ such that $t(A) = 4$;
        \item[v.] $t(A) = 5$ if and only if $A=J_1\oplus A(1)$;
        \item[vi.] $t(A) = 6$ if and only if $A=E(3)$;
        \item[vii.] $t(A) = 7$ if and only if $A=J_1\oplus A(2)$;
        \item[viii.] $t(A) = 8$ if and only if $A= E(3)\oplus A(1)$, $E(4)$, or \begin{align*}
            C_3 = \langle x,z,z' ~:~ xx=z,~xz = zx = z'\rangle;
        \end{align*}
        \item[ix.] $t(A) = 9$ if and only if $A=J_1\oplus A(3)$;
        \item[x.] $t(A) = 10$ if and only if $A=E(3)\oplus A(2)$, $E(4)\oplus A(1)$, $E(5)$.
    \end{enumerate} Here, $A(n)$ denotes the abelian algebra of dimension $n$ and $E(n)$ ranges over all extra special algebras of dimension $n$.
\end{thm}

\begin{proof} Throughout this proof, let $A$ be a nilpotent associative algebra of finite dimension $n$ and denote $m=\dim(A')$.

\textbf{Case $t(A)<2$.} We have already mentioned that $t(A)=0$ if and only if $A$ is abelian. If $t(A) = 1$, then Lemma \ref{derived bound} yields $1\geq m(m+1)$, which implies that $\dim(A') = 0$. But this means that $A$ is abelian, or $t(A) = 0$, a contradiction. Thus, there is no $A$ such that $t(A) = 1$.

\textbf{Case $t(A)=2$.} If $t(A) = 2$, then $2\geq m(m+1)$ yields $m=0$ or $m=1$. Again, $m=0$ leads to a contradiction since $A$ cannot be abelian. We thus assume that $\dim(A') = 1$. Either $Z(A)\subseteq A'$ or not. If not, we invoke Lemma \ref{ideal equality}, which guarantees a 1-dimensional ideal $Z$ and an ideal $I$ such that $A=I\oplus Z$ and $t(I) + 2\dim(I') = t(A) = 2$. If $\dim(I') = 0$, then $t(I) = 2$. But this also means that $I$ is abelian, and so this is a contradiction. If $\dim(I') = 1$, then $t(I) =0$. But this means that $I$ is abelian, a contradiction with $\dim(I') = 1$. So $Z(A)\not\subseteq A'$ leads to nothing. If $Z(A)\subseteq A'$, let $Z$ be a 1-dimensional ideal contained in $Z(A)$ and denote $H=A/Z$. Since $Z\subseteq A'$, we have $\dim H - \dim(H') = \dim A - \dim(A')$. Our inequality (\ref{bound i}) yields \[\dim M(A) + 1 \leq \dim(A/Z) + 2\dim(A/A'),\] or $n^2 - t(A) + 1\leq (n-1)^2 - t(H) + 2(n-\dim(A'))$, which simplifies to \begin{align*}
    t(H) + 2\dim(A')\leq t(A).
\end{align*} This implies that $t(H) \leq 0$. So $t(H)=0$, and $H$ is abelian, meaning $H'=0$. This implies that $A'\subseteq Z$, which yields $A' = Z(A) = Z$; in other words, $A$ is extra special. However, there is no such algebra with $t(A) = 2$.

\textbf{Case $t(A) = 3$.} We again start with $3\geq m(m+1)$, which guarantees that $m=1$ since $m=0$ leads to a contradiction. If $Z(A)\not\subseteq A'$, we invoke Lemma \ref{ideal equality}; again, this ensures that $A=I\oplus Z$ with $\dim Z = 1$ and $t(I) + 2\dim(I') = t(A) = 3$. If $\dim(I') = 0$, then $I$ is abelian and $t(I)=3$, a contradiction. If $\dim(I') = 1$, then $t(I) = 1$. However, as seen in the previous cases, there is no such algebra $I$. If $Z(A)\subseteq A'$, we again let $Z$ be a 1-dimensional ideal in $Z(A)$ and denote $H=A/Z$. Our inequality (\ref{bound i}) yields $t(H)\leq 1$. There is no algebra $H$ such that $t(H) =1$, so let us consider the case $t(H)=0$. Here, $H$ must be abelian, and so $A$ is extra special. The only extra special algebra with $t(A) = 3$ is $J_1$.

\textbf{Case $t(A) = 4$.} Here, Lemma \ref{derived bound} yields $4\geq m(m+1)$ and so, again, $\dim(A')$ must be 1. If $Z(A)\not\subseteq A'$, then $A=I\oplus Z$ with $t(I) + 2\dim(I') = 4$ by Lemma \ref{ideal equality}. If $\dim(I') = 0$, then $I$ is abelian and $t(I) = 4$, a contradiction. If $\dim(I') = 1$, then $t(I)=2$, but there is no such algebra $I$. If $\dim(I') = 2$, then $t(I) = 0$, but this means that $I$ is abelian, a contradiction with $\dim(I')\neq 0$. Suppose $Z(A)\subseteq A'$. We again let $Z$ be a 1-dimensional ideal in $Z(A)$ and denote $H=A/Z$. This time, our inequality (\ref{bound i}) yields $t(H)\leq 2$. There is no algebra $H$ such that $t(H)=1$ or $t(H) = 2$. In the case of $t(H)= 0$, one may deduce that $A$ is extra special. However, there is no extra special algebra $A$ such that $t(A) = 4$.

\textbf{Case $t(A)=5$.} We begin with $5\geq m(m+1)$ and again make the deduction that $\dim(A')$ must be 1. For $Z(A)\not\subseteq A'$, we have $A=I\oplus Z$ with $t(I) + 2\dim(I') = 5$. The case of $\dim(I') = 0$ leads to a contradiction. The case of $\dim(I') = 2$ forces $t(I) = 1$, but there is no such $I$. The case of $\dim(I') = 1$, however, implies that $t(I) = 3$. There is one algebra in this case: $I = J_1$. Therefore, $A = J_1\oplus A(1)$. This is consistent with $t(A) = 5$ since \begin{align*}
    \dim M(J_1\oplus A(1)) &= \dim M(J_1) + \dim M(A(1)) + 2\dim(J_1/J_1'\otimes A(1)) = 1 + 1 + 2(1)(1)
\end{align*} by (\ref{kunneth}), and so $t(A) = 3^2 - 4$. Now suppose that $Z(A)\subseteq A'$. Let $Z$ be a 1-dimensional ideal in $Z(A)$ and $H=A/Z$. This time, our inequality (\ref{bound i}) yields $t(H)\leq 3$. From previous cases, there is nothing for $t(H)=1$ or $t(H) = 2$. For $t(H)=0$, $A$ must be extra special. However, there is no extra special algebra $A$ such that $t(A)=5$. For $t(H)=3$, the only possibility is $H=J_1$, and so $A$ must be a central extension of $Z$ by $J_1$. In other words, there is a central extension \[0\xrightarrow{} Z\xrightarrow{} A\xrightarrow{} J_1\xrightarrow{} 0\] such that $Z\subseteq A'$. A basis for $A$ is $\{x,z,z'\}$ where $\{z'\}$ is a basis for $Z$ and $A$ has multiplications $xx = z$ and $xz = zx = z'$. But we need $A'$ to be 1-dimensional, and so this is not possible here. This extension will reappear in the case of $t(A)=7$, where we will consider it in more depth.

\textbf{Case $t(A) = 6$.} By our inequality $6\geq m(m+1)$, we deduce that $\dim(A')$ must equal 1 or 2 since $A$ is not abelian. Suppose $Z(A)\not\subseteq A'$. We have $A=I\oplus Z$ with $t(I) + 2\dim(I') = 6$. The cases of $\dim(I') = 0$ and $\dim(I') = 3$ lead to a contradiction with $I$ being abelian. If $\dim(I') = 1$, then $t(I) = 4$. If $\dim(I') = 2$, then $t(I) = 2$. In both of these cases, there is no such $I$. Suppose now that $Z(A) \subseteq A'$. Let $Z$ be a 1-dimensional ideal in $Z(A)$ and $H=A/Z$. Our inequality (\ref{bound i}) becomes $t(H) + 2\dim(A')\leq 6$. \begin{enumerate}
    \item[i.] If $\dim(A') = 1$, then $t(H)\leq 4$. There is no algebra such that $t(H) = 1,2,4$. If $t(H)=3$, then $H$ must equal $J_1$, making $A$ a central extension of $Z$ by $J_1$. But this would again force $\dim(A') \neq 1$ since $Z$ is also contained in $A'$, a contradiction. If $t(H)=0$, we have abelian $H$ and extra special $A$. Since $t(A)=6= 2\dim(A)$, we have $\dim(A) = 3$, and so $A$ may be any of $J_1\ast J_1$, $J_2$, $\Gamma_2$, or $H_2(\lambda)$, where $0\neq \lambda\neq 1$. In other words, $A=E(3)$.
    \item[ii.] If $\dim(A') = 2$, then $t(H)\leq 2$. There is nothing for $1$ or 2. If $t(H) = 0$, then $A$ is extra special. However, we are assuming that $\dim(A')=2$, and so this is a contradiction.
\end{enumerate}

\textbf{Case $t(A) = 7$.} By our inequality $7\geq m(m+1)$, we again deduce that $\dim(A')$ must equal 1 or 2. If $Z(A)\not\subseteq A'$, we have $A=I\oplus Z$ with $t(I) + 2\dim(I') = 7$. The case of $\dim(I') = 0$ leads to a contradiction with $t(I) = 7$. If $\dim(I') = 1$, then $t(I) = 5$, which implies that $I=J_1\oplus A(1)$. In this case, we obtain $A= J_1\oplus A(2)$ since $Z\subseteq Z(A)$ is 1-dimensional. If $\dim(I') = 2$, then $t(I) = 3$, which implies that $I=J_1$. However, $\dim(J_1') = 1$, a contradiction. If $\dim(I') = 3$, then $t(I) = 1$; there is nothing here. Suppose now that $Z(A) \subseteq A'$. Let $Z$ be a 1-dimensional ideal in $Z(A)$ and $H=A/Z$. Our inequality (\ref{bound i}) becomes $t(H) + 2\dim(A')\leq 7$.
\begin{enumerate}
    \item[i.] If $\dim(A') = 1$, then $t(H)\leq 5$. There is nothing for $t(H)=1,2,4$. If $t(H)=0$, then $H$ is abelian and $A$ is extra special; however, there are no extra special algebras with $t(A) = 7$. If $t(H) = 3$, then $H=J_1$. But $A'$ is 1-dimensional, and so this is a contradiction since $Z\subseteq A'$. If $t(H) = 5$, then $H=J_1\oplus A(1)$, which again conflicts with $\dim(A') = 1$.
    \item[ii.] If $\dim(A') = 2$, then $t(H)\leq 3$. There is nothing for $t(H) = 1,2$. If $t(H)=0$, then $A$ is extra special, which conflicts with $\dim(A')=2$. The final possibility of the case $t(A) = 7$ is if $t(H) = 3$. Here, we know that $H=J_1$ and that $A$ is a central extension of $Z$ by $J_1$. As before, a basis for $A$ is $\{x,z,z'\}$ where $\{z'\}$ is a basis for $Z$ and $A$ has multiplications $xx = z$ and $xz = zx = z'$. So far, this $A$ works; we note that $\dim(A')=2$. We now compute the multiplier $M(A)$.

    Let $\{m_i\}_{i=1,\dots,9}$ be a generating set for $M(A)$ and consider the following multiplication table for the cover of $A$. \begin{align*}
        &xx = z+m_1 && xz = z'+m_2 && xz' = m_3 \\ &zx = z'+m_4 && zz = m_5 && zz' = m_6 \\ &z'x = m_7 && z'z = m_8 && z'z' = m_9
    \end{align*} By a change of variables, we allow $m_1 = m_2 = 0$. Moreover, we compute $m_4 = m_6 = m_8=m_9 = 0$ and $m_3 = m_5 = m_7$ via the associative identity. The multiplications on the cover are therefore \begin{align*}
        &xx = z && xz = z' && xz' = m \\ &zx = z' && zz = m && zz' = 0 \\ &z'x = m && z'z = 0 && z'z' = 0
    \end{align*} where $m$ denotes the single basis element of the multiplier. Thus, $t(A) = 3^2 - 1 = 8$. This $A$ does not fit $t(A)=7$, but the work of computing $t(A)$ will come in handy for the next case.
\end{enumerate}

\textbf{Case $t(A)=8$.} We continue in the same fashion with $8\geq m(m+1)$. The case of $Z(A)\not\subseteq A'$ yields a set of four possible structures on $A=I\oplus Z$; these occur when $t(I) = 6$ and are based on that case. They are $(J_1\ast J_1)\oplus A(1)$, $J_2\oplus A(1)$, $\Gamma_2\oplus A(1)$, and $H_2(\lambda)\oplus A(1)$. The case of $Z(A)\subseteq A'$, when $\dim(A') = 1$, yields $A=E(4)$. When $\dim(A') = 2$, we obtain the central extension of $A(1)$ by $J_1$ from the $t(A)=7$ discussion. This $A$ has basis $\{x,z,z'\}$, multiplications $xx=z$ and $xz = zx = z'$, and $t(A)=8$. We call this algebra $C_3$ since it is generated by one element and is thus \textit{cyclic} in the Leibniz-algebraic sense.

\textbf{Case $t(A) = 9$.} We still have $\dim(A') = 1$ or 2. The case of $Z(A)\not\subseteq A'$ yields $A=J_1\oplus A(3)$. We note that this is consistent with \begin{align*}
    \dim M(A) &= \dim M(J_1) + \dim M(A(3)) + 2\dim(J_1/J_1' \otimes A(3)) = 1 + 9 + 2(1)(3) = 16
\end{align*} since $t(A) = 25 - 16 = 9$. When $Z(A)\subseteq A'$, we obtain nothing but contradictions (including another path to $C_3$) except, possibly, for a central extension \[0\xrightarrow{} Z\xrightarrow{} A\xrightarrow{} J_1\oplus A(1)\xrightarrow{}0\] of $Z=A(1)$ by $J_1\oplus A(1)$. We must compute all such $A$ and their multipliers.

Let $\{x,z,a,z'\}$ be a basis for $A$ where $x$ and $z$ are the usual elements of $J_1$, $z'$ generates $Z$, and $a$ generates the other $A(1)$. Using change of bases and the associative identity, we compute a general structure on $A$ to be \begin{align*}
    & xx=z && xz = \beta z' && xa = \alpha_1z' \\ & zx = \beta z' && ax = \alpha_2z' && aa = \alpha_3 z'
\end{align*} where at least one of the $\alpha_i$'s must be nonzero since $a\not\in A'$ and $Z(A)\subseteq A'$, forcing $a$ to be noncentral. In every one of these cases, the multiplier is found to be 4-dimensional, and thus $t(A) = 12$ for any central extension of $A(1)$ by $J_1\oplus A(1)$. This concludes our discussion of the $t(A) = 9$ case, but we include the computation of the multiplier for the isomorphism class \[\langle x,z,a,z'~:~ xx = z,~ ax = z'\rangle\] as an example of this work.

We begin with a general multiplication table for the cover.
\begin{align*}
    &xx = z+m_{11} && xz = m_{12} && xa = m_{13} && xz' = m_{14} \\ & zx = m_{21} && zz = m_{22} && za = m_{23} && zz' = m_{24} \\ & ax = z' + m_{31} && az = m_{32} && aa = m_{33} && az' = m_{34} \\ & z'x = m_{41} && z'z = m_{42} && z'a = m_{43} && z'z' = m_{44}
\end{align*} By a change of basis, we allow $m_{11}=m_{31} = 0$. By the associative identity, we obtain $m_{12} = m_{21}$, $m_{32} = m_{41}$, and $m_{14} = m_{22} = m_{23} = m_{24} = m_{34} = m_{42} = m_{43} = m_{44} = 0$. Attempting to apply the associative identity to $m_{13}$ and $m_{33}$ goes nowhere. We have thereby found all linear relations among the generating elements of the multiplier, and so a basis is $\{m_{12}, m_{13}, m_{32}, m_{33}\}$.

\textbf{Case $t(A) = 10$.} The inequality $10\geq m(m+1)$ yields $\dim(A') = 1$ or 2. When $Z(A)\not\subseteq A'$, we obtain $A=E(3)\oplus A(2)$ and $E(4)\oplus A(1)$. In the case $Z(A)\subseteq A'$, we obtain $E(5)$ when $\dim(A') = 1$. When $\dim(A')=2$, everything is immediately contradictory with the possible exception of a central extension \[0\xrightarrow{} A(1)\xrightarrow{} A\xrightarrow{} E(3)\xrightarrow{} 0\] of $A(1)$ by any of $J_1\ast J_1$, $J_2$, $\Gamma_2$, or $H_2(\lambda)$. We must compute all such $A$ and their multipliers. There are many possibilities to consider.
\begin{enumerate}
    \item[i.] When $A(1)$ is extended by $J_1\ast J_1 = \langle x,y,z ~:~ xx = yy = z\rangle$, we can choose a basis for $A$ with the multiplication structure
\begin{align*}
    & xx = z && xy = \alpha_1 z' \\ &yx = \alpha_2 z' && yy = z + \alpha_3 z'
\end{align*} where $z'$ generates $A(1)$ and at least one $\alpha_i$ is nonzero since $z'\in Z(A)\subseteq A'$. In every case, the multiplier is 4-dimensional, and so $t(A)=12$.

\item[ii.] When $A(1)$ is extended by $J_2 = \langle x,y,z~:~ xy = z\rangle$, we can choose a basis for $A$ with only nonzero multiplications \begin{align*}
    & xx = \alpha_1 z' && xy = z \\ & yx = \alpha_2 z' && yy = \alpha_3 z'
\end{align*} where $z'$ generates $A(1)$ and at least one $\alpha_i$ is nonzero. Here, $t(A)=12$ in all cases with the exception of the isomorphism class \[\langle x,y,z,z'~:~ xx = z',~ xy=z,~yx=z'\rangle,\] for which $t(A) = 13$.

\item[iii.] When $A(1)$ is extended by $\Gamma_2$, we obtain $t(A)=12$ in every case.

\item[iv.] When $A(1)$ is extended by $H_2(\lambda)$, $0\neq \lambda\neq 1$, we again obtain $t(A)=12$ in all cases.
\end{enumerate}
This concludes our discussion of the $t(A) = 10$ case as well as the proof of the main result.
\end{proof}


\newpage
\begin{thebibliography}{}

\bibitem{arab} Arabyani, H.; Niroomand, P.; Saeedi, F. ``On dimension of Schur multiplier of nilpotent Lie algebras II." \textit{Asian-European Journal of Mathematics}, Vol. 10, No. 04 (2017).

\bibitem{batten mult} Batten, P.; Moneyhun, K.; Stitzinger, E. ``On characterizing nilpotent Lie algebras by their multipliers." \textit{Communications in Algebra}, Vol. 24, No. 14 (1996).

\bibitem{batten} Batten, P. ``Multipliers and Covers of Lie Algebras." Dissertation, North Carolina State University (1993).

\bibitem{edal mult new} Edalatzadeh, B.; Hosseini, S. N. ``Characterizing nilpotent Leibniz algebras by a new
bound on their second homologies." \textit{Journal of Algebra}, Vol. 511 (2018).

\bibitem{edal mult} Edalatzadeh, B.; Hosseini, S. N.; Salemkar, A. R. ``Characterizing nilpotent Leibniz algebras by their multiplier." \textit{Journal of Algebra}, Vol. 578 (2021).

\bibitem{edal} Edalatzadeh, B.; Pourghobadian, P. ``Leibniz algebras with small derived ideal." \textit{Journal of Algebra}, Vol. 501, No. 1 (2018).

\bibitem{hardy} Hardy, P. ``On characterizing nilpotent Lie algebras by their multipliers. III." \textit{Communications in Algebra}, Vol. 33, No. 11 (2005).

\bibitem{hardy stitz} Hardy, P.; Stitzinger, E. ``On characterizing nilpotent Lie algebras by their multipliers, $t(L) = 3, 4, 5, 6$." \textit{Communications in Algebra}, Vol. 26, No. 11 (1998).

\bibitem{mainellis extra special} Mainellis, E. ``Associative Algebras with Small Derived Ideal'' (2022). arXiv:2212.01425

\bibitem{mainellis batten di} Mainellis, E. ``Multipliers and Unicentral Diassociative Algebras." \textit{Journal of Algebra and Its Applications} (2022). DOI: 10.1142/S0219498823501074

\bibitem{mainellis batten} Mainellis, E. ``Multipliers and Unicentral Leibniz Algebras.'' \textit{Journal of Algebra and Its Applications}, Vol. 21, No. 12 (2022).

\bibitem{mainellis batten tri} Mainellis, E. ``Multipliers and Unicentral Triassociative Algebras'' (2022). arXiv:2210.14110

\bibitem{mainellis nilp mult} Mainellis, E. ``Multipliers of Nilpotent Diassociative Algebras." \textit{Results in Mathematics}, Vol. 77, No. 5 (2022).

\bibitem{shamsaki} Niroomand, P.; Shamsaki, A. ``On characterizing nilpotent Lie algebras by their multiplier, $s(L)=4$." \textit{Rendiconti del Circolo Matematico di Palermo Series 2}, Vol. 69 (2020).

\end{thebibliography}
\end{document}